\documentclass[reqno,11pt]{amsart}
 \usepackage{amsmath, amsthm, a4, latexsym, amssymb,color}
\usepackage[unicode]{hyperref}
\usepackage{srcltx}

\def\@rmrk#1#2{\refstepcounter
    {#1}\@ifnextchar[{\@yrmrk{#1}{#2}}{\@xrmrk{#1}{#2}}}

\makeatletter\@addtoreset{equation}{section}\makeatother

 \newfont{\bfit}{cmbxti10 scaled 1200}

\renewcommand{\d}{{\rm d}}

 \newcommand{\e}{{\rm e} }

 \newcommand{\eps}{\varepsilon}

 \newcommand{\R}{\mathbb{R}}
 \newcommand{\N}{\mathbb{N}}
 \newcommand{\Z}{\mathbb{Z}}

 \newcommand{\E}{\mathbb{E}}
 \renewcommand{\P}{\mathbb{P}}
 \newcommand{\Q}{\mathbb{Q}}
 \def\1{{\mathchoice {1\mskip-4mu\mathrm l} 
{1\mskip-4mu\mathrm l}
{1\mskip-4.5mu\mathrm l} {1\mskip-5mu\mathrm l}}}

\renewcommand{\subsection}{\secdef \subsct\sbsect}
\newcommand{\subsct}[2][default]{\refstepcounter{subsection}
\vspace{0.15cm}
{\flushleft\bf \arabic{section}.\arabic{subsection}~\bf #1  }
\nopagebreak\nopagebreak}
\newcommand{\sbsect}[1]{\vspace{0.1cm}\noindent
{\bf #1}\vspace{0.1cm}}

{\nopagebreak {\hfill\rule{2mm}{2mm}}\\ }

\newtheorem{theorem}{Theorem}[section]
\newtheorem{lemma}[theorem]{Lemma}
\newtheorem{cor}[theorem]{Corollary}
\newtheorem{prop}[theorem]{Proposition}

\newtheoremstyle{thm}{1.5ex}{1.5ex}{\itshape\rmfamily}{}
{\bfseries\rmfamily}{}{2ex}{}

\newtheoremstyle{rem}{1.3ex}{1.3ex}{\itshape\rmfamily}{}
{\itshape\rmfamily}{}{1.5ex}{}
\theoremstyle{rem}
\newtheorem{remark}{{\slshape\rmfamily Remark}}[]

\refstepcounter{subsubsection}

\def\thebibliography#1{\section*{References}
  \list%
  {\arabic{enumi}.}
    {\settowidth\labelwidth{[#1]}\leftmargin\labelwidth
    \advance\leftmargin\labelsep
    \parsep0pt\itemsep0pt
    \usecounter{enumi}}
    \def\newblock{\hskip .11em plus .33em minus .07em}
    \sloppy                   
    \sfcode`\.=1000\relax}

\begin{document}
 \title[Weak and strong disorder for the SHE in $d\geq 3$]{Weak and strong disorder for the stochastic heat equation and 
	continuous directed
polymers in $d\geq 3$}
\author{Chiranjib Mukherjee}
\address{Courant Institute, NYU and WIAS Berlin} 
\email{mukherjee@cims.nyu.edu}
\author{Alexander Shamov}
\address{Weizmann Institute of Science.}
\email{trefoils@gmail.com}
\author{Ofer Zeitouni}
\address{Weizmann Institute of Science and Courant Institute.}
\thanks{Partially supported by a grant from the Israel Science Foundation and by a US-Israel BSF grant}
\email{ofer.zeitouni@weizmann.ac.il}
\date{January 7, 2015}
\maketitle

\centerline{\textit{Courant Institute  and WIAS Berlin, Weizmann Institute, Weizmann Institute and Courant Institute }}




\begin{abstract}
We consider the smoothed
multiplicative noise stochastic heat equation 
$$\d u_{\eps,t}= 
\frac 12 \Delta u_{\eps,t} \d t+ 
\beta \eps^{\frac{d-2}{2}}\, \, u_{\eps, t} \, \d B_{\eps,t} ,
\;\;u_{\eps,0}=1,$$
in dimension
$d\geq 3$, where $B_{\eps,t}$ is a spatially smoothed (at scale $\eps$)
space-time white noise, and $\beta>0$ is a parameter.
We show the existence of a $\bar\beta\in (0,\infty)$ so 
that the solution exhibits weak disorder when $\beta<\bar\beta$ and
strong disorder when $\beta > \bar\beta$. The proof techniques use elements of the theory of the Gaussian multiplicative chaos.
\end{abstract}
\section{Motivation and introduction}\label{sec-intro}

We consider the stochastic heat equation (SHE) with multiplicative noise,
written formally as
\begin{equation}\label{stheateq}
 \partial_t  u(t,x)= \frac 12 \Delta  u(t,x) +  u(t,x) \, \eta (t,x). 
\end{equation}
Here $\eta$ is the ``space-time white noise", which formally is
the centered Gaussian process with
covariance function 
$\E(\eta(s,x) \eta(t,y))= \delta_0(t-s)\delta_0(x-y)$ for $s,t>0$ and 
$x,y \in \R^d$. We emphasize that \eqref{stheateq} is a formal expression,
and in attempting to give it a precise meaning one is immediately faced
with the problem of multiplication of distributions.


Besides the intrinsic interest in the SHE, we recall that
the Cole-Hopf transformation $h:= -\log u$ formally transforms the SHE 
to the non-linear Kardar-Parisi-Zhang (KPZ) equation, 
which 
can be written as
\begin{equation}\label{KPZ}
\partial_t h(t,x)= \frac 12 \Delta h(t,x)- 
\frac 1 2 (\partial_ x h(t,x))^2 + \eta,
\end{equation}
and appears in dimension 
$d=1$ as the limit of front propagation in certain exclusion processes
(\cite{BG97}, \cite{ACQ11}). While a-priori the equation
\eqref{KPZ} is not well posed due to the presence of products of distributions,
much recent progress has been achieved in giving  an intrinsic precise interpretation to it in dimension $d=1$
(\cite{H13})

As discussed in \cite{AKQ} and 
\cite{CSZ15},
the equations \eqref{stheateq} and \eqref{KPZ} share 
close analogies to the well-studied 
{\it{discrete directed polymer}}, which can 
be defined as the transformed path measure
\begin{equation}\label{discretepolym}
\mu_n(\d \omega)= \frac 1{Z_n} \exp\bigg\{\beta \sum_{i=1}^n \eta(i,\omega_i)\bigg\} \d P_0.
\end{equation}
Here the white noise (the {\it{disorder}}) is replaced by i.i.d. random variables $\eta=\{\eta(n,x)\colon n\in\N, x\in \Z^d\}$, $P_0$ denotes the law of a simple random walk starting at the origin
corresponding to a $d$-dimensional path $\omega_n=(\omega_i)_{i\leq n}$, while $\beta>0$ stands for the strength of the disorder. 
It is well-known that, when $d\geq 3$  the
normalized partition function
$Z_n / \E Z_n$ converges almost surely to a 
random variable $Z_\infty$, which, when $\beta$ is small enough, is positive almost surely (i.e., {\it{weak disorder}} persists \cite{IS88, B89}), \
while for $\beta$ large enough, $Z_\infty=0$ (i.e.,
{\it{strong disorder holds}} \cite{CSY04}). 
  Related results
for a continuous directed polymer in a field
of random traps appear in \cite{CY13}.

We return
to the study of the stochastic heat equation in the continuum $\R^d$, written as a stochastic differential equation
\begin{equation}\label{stheateqSDE}
\d u_t= \frac 12 \Delta u_t \d t+ \beta\, u_t\, \d B_t,
\end{equation}
where $B_t$ is a cylindrical Wiener process in $L^2(\R^d)$.
Since the solution to 
\eqref{stheateqSDE} is not well defined,
a standard approach to treat this equation
is to introduce a regularization of the 
process $B_t$, followed by a suitable rescaling of 
the coupling coefficients and subsequently 
passing to a limit as the regularization is turned off. 
In one space dimension $d = 1$,
this task was carried out by Bertini-Cancrini (\cite{BC95}) by  
expressing the regularized process by a Feynman-Kac formula; 
after a simple renormalization 
({\it{the Wick exponential}}), a meaningful expression 
was obtained when the mollification was removed. 
The renormalized Feynman-Kac formula defines a process 
with continuous (in space and time) trajectories and
it solves the equation \eqref{stheateqSDE} 
(when the stochastic differential is interpreted in the Ito sense).  
Extending this procedure to $d = 2$ (where small scale singularities coming from the noise are stronger), Bertini-Cancrini (\cite{BC98}) 
introduced a rescaling of the coupling constant 
  $$
\beta=\beta(\eps)= \left( \frac{2\pi}{\log \eps^{-1}}+ \frac {C}{(\log \eps^{-1})^2} \right)^{1/2} \qquad C\in \R
$$
which vanishes as $\eps\to 0$. It turned out that the covariance $\E[Z_\eps(t,x) Z_\eps(t,y)]$ of the regularized field $Z_\eps$ converges to a non-trivial limit as the mollification is removed,
but the limiting law of $Z_\eps$ was not identified in \cite{BC98}. 
The latter identification was recently 
carried out by Caravenna, Sun and Zygouras (\cite{CSZ15})  and by
Feng \cite{Feng},
who proved that, in $d=2$,
if $\beta_\eps$ is chosen to be $\beta \sqrt{2\pi\,\,[\log(1/\eps)]^{-1}}$, then for $\beta<1$, $Z_\eps$ converges in law to a random variable with an explicit distribution,
while for $\beta\geq 1$, $Z_\eps$ converges in law to $0$.

The results of this article concern related 
statements for $d\geq 3$ pertaining to the smoothened and rescaled equation
$$
\begin{aligned}
& \d u_{\eps,t}= \frac 12 \Delta u_{\eps,t}+ \beta\, \eps^{\frac{d-2}2}\, \, u_{\eps,t} \, \d B_{\eps,t}\\
& u_{\eps,0}= 1
 \end{aligned}
$$
Write $u_\eps(x):=u_{\eps,1}(x)$.
Our main result shows that for every $x\in \R^d$, for any $\beta$ small enough
$u_\eps(x)$ converges in distribution
to a non-degenerate random variable $Z_\infty=Z_\infty(\beta)$, i.e., {\it{weak disorder}} prevails, 
while for $\beta$ large enough, $u_\eps(x)$ converges in probability to $0$, i.e., {\it{strong disorder}} takes place. 
We also show that for 
$\beta$ small enough and any suitable test function 
$f$, $u_\eps(f)=\int  f(x) u_\eps(x)\d x$ converges in probability 
to $\int  f(x) \ \d x$. We remark that our results, 
unlike \cite{CSZ15}, do not charaterize the limiting non-degenerate random variable $Z_\infty(\beta)$,
nor do they
identify the exact critical threshold for the value of $\beta$ (which happens to be $1$ in $d=2$), where the departure
from weak disorder to strong disorder takes place.

\section{Main results.} \label{sec-results}
\subsection{Preliminaries.}
We consider a complete probability space $(\Omega,\mathcal F, \P)$ and a 
cylindrical Wiener process $B=(B_t)_{t\geq 0}$ on $L^2(\R^d)$. 
The latter is defined as the  centered Gaussian process with covariance 
$$
\E\bigg(B_s(f)B_t(g)\bigg) = \big(s\wedge t\big) \, \int_{\R^d} f(x) g(x)\d x \qquad f,g \in \mathcal S(\R^d).
$$
Here $\mathcal S=\mathcal S(\R^d)$ is the Schwartz space of rapidly decreasing functions in $\R^d$. To define $B$ pointwise in $\R^d$, we need the regularization
$$
B_{\eps,t}(x)= B_t\big(\phi_\eps(x- \cdot)\big),
$$
with respect to some mollifier 
$$
\phi_\eps= \eps^{-d} \phi(x/\eps).
$$
Here $\phi$ is some smooth, non-negative, compactly supported and even function such that $\int_{\R^d} \phi(x)\d x =1$. Then $\int_{\R^d} \phi_\eps(x)\d x =1$, and $\phi_\eps\Rightarrow\delta_0$ weakly as probability measures. Furthermore, for any $\eps>0$, $B_\eps=(B_{\eps,t})_{t\geq 0}$ is also a centered Gaussian process with covariance 
$$
\E\big(B_{s,\eps}(x)B_{t,\eps}(y)\big)= \big(s\wedge t\big) V_\eps(x-y) 
$$
where we introduced
\begin{equation}
  \label{thursday1}
 V=\phi\star \phi, V_{\eps,\delta}=
  \phi_\eps\star \phi_\delta,
  V_\eps=V_{\eps,\eps}.
\end{equation}
Note that $V_\eps(x)=\eps^{-d}V(x/\eps)$.

For any $\beta>0$ and $\eps>0$, we
consider the stochastic differential equation
\begin{equation}\label{regspde}
\begin{aligned}
	&\d u_{\eps,t}= \frac 12 \Delta u_{\eps,t} \d t+ \beta \eps^{\frac{d-2}{2}}\, \, u_{\eps, t} \, \d B_{\eps,t} \\
&u_{\eps,0}=1,
\end{aligned}
\end{equation}
where the stochastic differential is interpreted in the classical Ito sense
(since our smoothing of $B$ was done in space only, the well-defined
solution
$u_{\eps,t}$ is adapted to the natural filtration $\mathcal{G}_t=
\sigma(\{B_{\eps,s}(x), x\in \R^d, s\leq t\}$). 
Our goal is to study the behavior of $u_{\eps,1}(x)$
as the mollification parameter $\eps$ is turned off. For this, we will
use
a convenient Feynman-Kac representation of $u_{\eps,t}(x)$, which we introduce in Section \ref{sec-FK-rep}
after stating our main results.

\subsection{Main results: Weak and strong disorder.}
Henceforth we fix $d\geq 3$ and set $u_\eps(x):=u_{\eps,1}(x)$ and,
for any $f\in \mathcal S(\R^d)$, we 
write $u_{\eps}(f)= \int_{\R^d}  u_\eps(x) f(x)\d x$. 
Here is the statement of our first main result. 
\begin{theorem}[Convergence to the heat equation in the
	weak disorder phase]\label{thm1}
There exists $\beta_\star\in(0,\infty)$ such that for all 
$\beta<\beta_\star$ and any $f\in\mathcal S(\R^d)$, $u_\eps(f)$ 
converges in probability to $\int_{\R^d} f(x) dx$ as $\eps\to 0$.
Furthermore, for any $\beta< \beta_\star$ and 
any $x\in\R^d$, $u_\eps(x)$ converges in distribution
to a random variable $Z_\infty$ which is positive almost surely.
 \end{theorem}
 \begin{remark}
	 The first statement in 
Theorem \ref{thm1} implies that 
$u_\eps$ converges in the sense of distributions 
to the solution of the heat equation.
Although for simplicity we content ourselves with the initial condition 
$Z_\eps(0,x)= 1$ in 
\eqref{regspde},
the same statement continues to hold for reasonably nice initial condition
$u_\eps(0,x)=u_0(x)$.
 \end{remark}
 \begin{remark}
	 While we do not discuss it in detail, the Feynman-Kac representation
	 of $u_\eps(x)$ that we introduce in the next subsection shows
	 that $u_\eps(x)$ and $u_\eps(y)$ become asymptotically 
	 independent as $\eps\to 0$; this explains the fact that smoothing 
	 with $f$ makes $u_\eps(f)$ deterministic.
 \end{remark}
 The proof of Theorem \ref{thm1} 
 is based on an $L^2$ computation and 
 is presented in Section \ref{sec-thm1}.

 \begin{theorem}[The strong disorder phase]\label{theo-UT}
 There is $\beta^*>0$ such that for all $\beta>\beta^*$,
  $u_\eps\to 0$ in probability.
\end{theorem}

The proof of Theorem \ref{theo-UT} is presented in Section \ref{sec-theo-UT}. 
This proof avoids the use of the well-known {\it{fractional moment method}} 
which pervades the proofs of strong disorder assertions in realm of 
the 
aforementioned literature on the discrete directed polymer models, 
and instead uses the theory of 
{\it{Gaussian multiplicative chaos}} (GMC).

As a by-product of our arguments, we have the following corollary.
\begin{cor}
	\label{cor-beta}
	There is a $\bar\beta\in (0,\infty)$ such that, as $\eps\to 0$,
	$u_\eps(0)$ converges
	to $0$ in probability for all $\beta > \bar \beta$ while
	$u_\eps(0)$ converges in distribution to a non-degenerate, strictly 
	positive random variable $Z_\infty=Z_\infty(\beta)$ 
	when $\beta<\bar \beta$.
\end{cor}
The constant $\bar\beta$ is given as the threshold for the uniform integrability
of a certain family of martingales
$Z_{\eps,\beta}$; we refer to the proof of Corollary \ref{cor-beta} for details, which can also be found at the end of Section \ref{sec-theo-UT}. We leave unresolved the question of what happens at $\beta=\bar\beta$.

\begin{remark} Clearly $\bar \beta$ depends on the dimension $d\geq 3$ and the mollifier $\phi$. As mentioned in Section \ref{sec-intro}, it
	remains an open problem to determine the exact value of $\overline \beta\in (0,\infty)$ and to identify
the exact distribution of the positive random variable $Z_\infty$ appearing in 
Corollary \ref{cor-beta}.
\end{remark}

\subsection{A Feynman-Kac representation.} \label{sec-FK-rep}
For any $x\in \R^d$, 
let $P_x$ denote the Wiener measure corresponding to a $d$-dimensional Brownian motion $(W_t)_{t\geq 0}$ starting at $x$ and independent of the cylindrical Wiener process $B$. 
$E_x$ will denote the corresponding expectation. 
For fixed $W$, set
 \begin{equation}\label{MIto}
 M_{\eps,t}(W)=  \int_0^t \int_{\R^d} \phi_\eps(W_s- x) \,\, \dot{B}(t-s,\d x)\d s 
 \end{equation}
as a Wiener integral.
For  
two fixed
$W$ and $W^\prime$, 
the  covariance is given by
\begin{equation}\label{quadvar}
\E\left(M_{\eps,t}(W)\cdot M_{\delta,t}(W^\prime)\right)= 
\int_0^t  \, V_{\eps,\delta}(W_s- W^\prime_s) \d s
\end{equation}
(recall \eqref{thursday1}. Here and later, we write $\E$ for integration over
$B$ only, keeping $W$ fixed).
We also note that, for any fixed $W$,
$$
\E\left(M_{\eps,t}^2(W)\right)=
t V_\eps(0)= t (\phi_\eps\star\phi_\eps)(0),
$$ 
which diverges like $\eps^{-d}$ as $\eps\to 0$.

We now turn to \eqref{regspde} and write its renormalized Feynman-Kac solution as
\begin{equation}\label{renormFK}
\begin{aligned}
	u_{\eps,t}(x)&=E_x \bigg[\exp\bigg\{ \beta \eps^{(d-2)/2} 
		M_{\eps,t}(W)\,\,- 
		\frac{\beta^2 \eps^{d-2}}{2}
	\E(M_{\eps,t}(W)^2)\bigg\}\bigg] \\
	&= E_x \bigg[\exp\bigg\{{\beta \eps^{(d-2)/2}} \int_0^t \int_{\R^d} \phi_\eps(W_s- x) \,\, \dot{B}(t-s,\d x) \d s
	  - \frac {\beta^2 \eps^{d-2}}2\,\, t V_\eps(0)\bigg\}\bigg].
\end{aligned}
\end{equation}
Note that 
$\E[u_{\eps,t}(x)]=1$.  

For our purposes, it is convenient to introduce another representation of $u_{\eps,t}$. 
Note that by rescaling of time and space, $\eps^{-1} W_s$ has the same distribution as $W_{s\eps^{-2}}$, while
$\dot{B}(s,\d x)\d s$ has the same distribution as 
  $$
\eps^{d/2+1} \dot B\big(s\eps^{-2}, \d \eps^{-1} x \big) \d s.
$$
Then, by \eqref{MIto}, for a fixed $W$,
 \begin{equation*}\label{martdef}
 \begin{aligned}
	 M_{\eps,t}(W)&\stackrel{(\mathrm d)}{=}\frac 1 {\eps^{(d-2)/2}} \int_0^{t\eps^{-2}} \int_{\R^d} \phi\big({y-\eps^{-1} W_{s \eps^2}}\big)\dot{B}(t/\eps^2-s,\d y) \d s \\
 \end{aligned}
 \end{equation*}
Hence \eqref{renormFK} implies  that
\begin{equation}\label{eq-trep}
\begin{aligned}
 u_{\eps,t}(x)\stackrel{(\mathrm d)}=E_{\frac x\eps}\bigg[\exp\bigg\{\beta \int_0^{t\eps^{-2}} \int_{\R^d} \phi\big({y-W_s}\big)\dot B(t/\eps^2-s, \d y) \d s - \frac {\beta^2}{2\eps^2} t V(0)\bigg\}\bigg]. 
 \end{aligned}
 \end{equation} 
 Recall that $u_\eps(x)=u_{\eps,1}(x)$.
 Using the invariance of the distribution of $\dot B$ under time reversal,
 we obtain that the spatially-indexed process $\{u_\eps(x)\}$ possesses the same
 distribution as the process 
   $\{Z_\eps(x/\eps)\}$,
 where 
 \begin{equation}\label{eq-ofer1}
\begin{aligned}
  Z_\eps(x)=E_{x}\bigg[\exp\bigg\{\beta \int_0^{\eps^{-2}} 
\int_{\R^d}\phi\big({y-W_s}\big)\dot B(s, \d y) \d s  - \frac {\beta^2}{2\eps^2} V(0)\bigg\}\bigg] 
 \end{aligned}
 \end{equation} 

\section{Proof of Theorem \ref{thm1}: The second moment method}\label{sec-thm1}
We start with an elementary computation.
\begin{lemma}\label{lemma1}
If $\beta>0$ is chosen small enough, for any $x\in \R^d$, the family $\{u_\eps(x)\}_{\eps>0}$ remains bounded in $L^2(\P)$.
\end{lemma}
\begin{proof}

Let $W$ and $W^\prime$ be two independent standard Brownian motions with 
$P_0 \otimes P_0$ denoting their joint law. Then,
writing $\eta_\eps=\eps^{(d-2)/2}$ and $M_\eps(W)=M_{\eps,1}(W)$,
  $$
\begin{aligned}
	\E\big[u_\eps(0)^2\big]&= \E\bigg[\bigg\{ E_0 \exp\bigg(\beta 
	\eta_\eps M_{\eps}(W)\,\,- \frac {\beta^2 \eta_\eps^2}2\,\,
V_\eps(0)\bigg)\bigg\}^2\bigg] \\
&= \big(E_0\otimes E_0\big) \,\, \bigg[\E\bigg\{\exp\bigg({\beta 
\eta_\eps M_{\eps}(W)\,\,- \frac {\beta^2 \eta_\eps^2}2\,\, 
 V_\eps(0)\bigg)\,\,\exp\bigg(\beta 
\eta_\eps} M_{\eps}(W^\prime)\,\,- \frac {\beta^2 \eta_\eps^2}2\,\,
 V_\eps(0)\bigg)\bigg\} \bigg]
\\
&= \big(E_0\otimes E_0\big) 
\bigg[\exp\bigg\{\beta^2\eta_\eps^2  \int_0^1 
V_\eps(W_s-W^\prime_s)\d s\bigg\}\bigg] 
=E_0 \bigg[\exp\bigg\{\beta^2\eta_\eps^2  \int_0^1 V_\eps(\sqrt 2W_s)
\d s\bigg\}\bigg]\\
\end{aligned}
$$
where the third identity follows by \eqref{quadvar}. 
Hence, by \eqref{thursday1}, 
Brownian scaling and change of variables, we infer that
$$
\E\big[u_\eps^2(0)\big] = E_0 \bigg[\exp\bigg\{\beta^2  \int_0^{1/\eps^2} V(\sqrt 2W_s)\d s\bigg\}\bigg]
\leq E_0 \bigg[\exp\bigg\{\beta^2  \int_0^{\infty} 
V(\sqrt 2W_s)\d s\bigg\}\bigg].
$$
Since $V$ is a bounded function of compact support, 
it is easy to check that for $\beta$ small enough, 
\begin{equation}\label{P1}
\sup_{x\in \R^d} E_x\bigg\{\beta^2 \int_0^\infty V(W_s) \d s \bigg\} \leq \eta <1.
\end{equation}
Hence, by Portenko's lemma (\cite{P76}), 
\begin{equation}\label{P2}
\sup_{x\in \R^d} E_x\bigg[\exp\bigg\{\beta^2 \int_0^\infty V(W_s)  \d s \bigg\}\bigg] \leq \frac 1 {1-\eta}<\infty.
\end{equation}
This proves the lemma.
\end{proof}

\begin{remark} Let us remark that $u_\eps$ is not a Cauchy sequence 
	in $L^2(\P)$ (which is
	the reason why the convergence in distribution in Theorem
	\ref{thm1} cannot be upgraded to convergence in 
	probability). 
	A simple computation using \eqref{quadvar} shows that
  $$
\begin{aligned}
&\E\big[(u_\eps- u_\delta)^2\big]\\
&=  E_0\otimes E_0\bigg[
 \exp\bigg\{\beta^2\eta_\eps^2 \, \int_0^t V_{\eps}\big(W_s- W^\prime_s\big) \d s \bigg\} -
 \exp\bigg\{\beta^2\eta_\eps \eta_\delta \,
 \int_0^t V_{\eps,\delta}\big(W_s- W^\prime_s\big) \d s\bigg\}\bigg]
\\
&+E_0\otimes E_0\bigg[\exp\bigg\{\beta^2\eta_\delta^2 \, \int_0^t V_{\delta}\big(W_s- W^\prime_s\big) \d s \bigg\} -
\exp\bigg\{\beta^2\eta_\eps\eta_\delta
 \, \int_0^t V_{\eps,\delta}\big(W_s- W^\prime_s\big) \d s \bigg\}
 \bigg]
\end{aligned}
$$
 The difference of the two terms in the first line (and likewise, the second line) does not go to zero. For instance,
 if $\phi_\eps$ is a centered Gaussian mollifier with variance $\eps^2$, then in the first line, again by Brownian scaling, the
 second term (with the expectation) becomes (recall \eqref{thursday1})
  $$
(E_0\otimes E_0) \bigg[\exp\bigg\{\beta^2 \,
\frac{\eta_\eps\eta_\delta}{\eta_{\sqrt{\eps^2+\delta^2}}^2} 
\int_0^{t/(\eps^2+\delta^2)} V\big(W_s- W^\prime_s\big) \d s \bigg\}\bigg]
 $$
 while the first term becomes 
   $$
(E_0\otimes E_0) \bigg[ \exp\bigg\{\beta^2 \, \int_0^{t/\eps^2} V\big(W_s- W^\prime_s\big) \d s \bigg\}\bigg].
 $$
 From these expressions one can see that $\E \big[ (u_\varepsilon - u_\delta)^2 \big]$ does not vanish, e.g., in the iterated limit
 $\lim_{\varepsilon \to 0} \lim_{\delta \to 0}$.
 \end{remark}
 
We turn to the proof of Theorem \ref{thm1}.
 
{\bf{Proof of Theorem \ref{thm1}.}} Let us denote by $\widehat u_\eps(x)= u_\eps(x) - \E(u_\eps(x)) = u_\eps(x)- 1$ and $\widehat u_\eps(f)= \int_{\R^d} f(x) \widehat u_\eps(x) \, \d x $.
Then $\E\big(\widehat u_\eps(f)\big)=0$.
Note that, for the proof of the first part of Theorem \ref{thm1}, it suffices to show that 
\begin{equation}\label{thm1pf0}
\E\big[\widehat u_\eps(f)^2\big] \to 0
\end{equation}
as $\eps\to 0$. Let us prove this fact.
Exactly similar computations as in the proof of Lemma \ref{lemma1} imply that
\begin{equation}\label{thm1pf}
 \begin{aligned}
\E\big[\widehat u_\eps(f)^2\big]&= \int \int_{\R^d\times\R^d}  f(x) f(y) \E\big[u_\eps(x) u_\eps(y)\big] \d x \d y
- \bigg(\int_{\R^d} f(x)\d x\bigg)^2\\
&= \int \int_{\R^d\times\R^d}  f(x) f(y)  
E_{\frac{x-y}\eps}\, \bigg[\e^{\frac{1}{2} \beta^2 \int_0^{2/\eps^2} V(W_s) \d s }\bigg] 
\d x \d y- \bigg(\int_{\R^d}  f(x)\d x\bigg )^2
\end{aligned}
\end{equation}
If $z=(x-y)/\eps$, then,
  \begin{equation}\label{P3}
E_z\bigg[\int_0^\infty V(W_s) \, \d s \bigg] = C_d \int \d y \, \frac{V(y)}{|y-z|^{d-2}} \to 0 \quad \mbox{as}\, |z|\to \infty.
\end{equation}
By applying Portenko's lemma again (\cite{P76}), we see that for $\beta$ small enough
\begin{equation}\label{P4}
\sup_x E_x \left[ \e^{\frac{\beta^2}{2} \intop_0^\infty V(W_s) \d s } \right] < \infty.
\end{equation}
Together with \eqref{P3}, by Lebesgue's convergence theorem, for an even smaller $\beta$ we have
\begin{equation}\label{P5}
E_z \left[ \e^{\frac{\beta^2}{2} \intop_0^\infty V(W_s) \d s } \right] \to 1
\end{equation}
as $|z| \to \infty$.
Combining \eqref{thm1pf}, \eqref{P4} and \eqref{P5}, we use the bounded convergence theorem to conclude \eqref{thm1pf0}. This proves the first part of Theorem \ref{thm1}.

For the second part, 
note that
\eqref{eq-trep} implies that for fixed $\eps$,
$u_{\eps,1}(0)$ is equal in distribution to
$Z_\eps$. Since the process $\{Z_\eps\}_\eps$ is a positive martingale
(with respect to a filtration indexed by $1/\eps^2$), it converges almost 
surely to a limit $Z_\infty$.
By Lemma \ref{lemma1}, $Z_\eps$ is  (uniformly in $\eps$) $L^2(\P)$ 
bounded for $\beta$ small enough, and therefore
$Z_\infty$ does not vanish identically. By the $0-1$ law as in the
proof of Theorem \ref{theo-UT} (see \eqref{0-1}), we conclude that $P(Z_\infty=0)=0$.
We conclude that $u_\eps(0)$ converges in distribution to $Z_\infty$. Further,
since $u_\eps(x)\stackrel{d}{=}u_\eps(0)$ by translation invariance, the same
applies to $u_\eps(x)$.
\qed

\section{Proof of Theorem \ref{theo-UT} and
Corollary \ref{cor-beta}: Gaussian multiplicative chaos}\label{sec-theo-UT}
The starting point is the representation \eqref{eq-ofer1} for 
$Z_\eps=Z_{\eps}(0)$. For $d\geq 3$, which we assume throughout,
we will show that there is a $\beta^*>0$ such that for all $\beta>\beta^*$,
$Z_\eps\to 0$ in probability.

In order to prove this result, we represent $Z_\eps$ as a 
Gaussian Multiplicative Chaos (GMC), see \cite{Kahane,shamov} for background. Let
$\mathcal{E}=C_0([0,\infty); \mathbb{R}^d)$ and recall that
  $P_0$ denotes the standard 
  Wiener measure on $\mathcal{E}$ corresponding to the 
  $d$-dimensional Brownian motion $W=(W_t)_{t\geq 0}$. Set 
  $$
  \Lambda_\eps= \exp\bigg\{\beta \int_0^{\eps^{-2}} \int_{\R^d}
   \phi\big({y-W_s}\big)\dot B(s, \d y)\d s
   - \frac {\beta^2}{2\eps^2} V(0)\bigg\}
$$
 and recall that $Z_\eps\stackrel{(d)}{=}E_0 \Lambda_\eps$. 
 Introduce the random measure 
$M_\eps$ with $\d M_\eps=\Lambda_\eps \, \d P_0$ on
$\mathcal{E}$ and note that $Z_\eps=\int_{\mathcal{E}}M_\eps(\d W)$.

Introduce the event 
  $\mathcal{V}:=\{Z_\eps\not\to_{\eps\to 0} 0\}$. 
Since $\mathcal{V}$ is a tail event for the process
$t\to B(t,\cdot)$, one has 
\begin{equation}\label{0-1}
\P(\mathcal{V})\in \{0,1\}.
\end{equation}

Note that
$\eps^{-1}\mapsto Z_\eps$ is a 
strictly positive martingale of mean $1$. Introduce on
$\Omega \times \mathcal{E}$ the measure
$$
\d \Q_\eps := \Lambda_\eps \, \d (\P\otimes P_0).
$$
Let the measure $\overline \Q_\eps$ be its marginal on $\Omega$, i.e.
$\d \overline \Q_\eps=Z_\eps \d \P$.
\begin{lemma}
  \label{lem-increase}
  If the sequence $(Z_\eps)_\eps$ is uniformly integrable under $\P$, then
  under $\overline\Q_\eps$, $(Z_\eps)_\eps$
  is uniformly bounded in probability. In other words, 
$$
\lim_{m\to\infty}\sup_\eps \overline\Q_\eps(Z_\eps>m)=0.
$$
\end{lemma}
\begin{proof}
Assume that $Z_\eps$ is uniformly integrable. Then, by the 
la Vall\'{e}e-Poussin theorem, there exists a 
convex increasing function 
$h:\R_+\to\R_+$, such that $h(x) / x \to \infty, x \to \infty$ and
$\sup_\eps \E h(Z_\eps)=C<\infty$. Then,
$$C\geq \E h(Z_\eps)=\int \frac{h(Z_\eps)}{Z_\eps}d\overline\Q_\eps \,.$$
The conclusion follows.
\end{proof}
\begin{remark} The implication in Lemma \ref{lem-increase} is an ``if and only if" statement; 
we only stated the direction that we need. 
\end{remark}

  Another preparatory  step that we need is the following proposition, whose
  statement and proof closely follow \cite[Prop. 3.1]{CY06}.
  \begin{prop}
    \label{prop-CY06}
    The sequence $\{Z_\eps\}$ is uniformly integrable under $\P$ if and only if
    $\P(\mathcal{V})=1$.
\end{prop}
  \begin{proof}
    If $\{Z_\eps\}$ is uniformly integrable under $\P$ then its limit is necessarily non degenerate, i.e. $\P(\mathcal{V})>0$. Then,  $\P(\mathcal{V})=1$ by 
    \eqref{0-1}.

    To prove the reverse implication,
    recall the random variables  $Z_\eps(x)$ (with $x\in \R^d$), 
    see \eqref{eq-ofer1}. With
    $t=1/\eps^2$, we write $\bar Z_t(x)=Z_\eps(x)$. It is enough
    to prove the uniform integrability for the sequence 
    $\bar Z_n(0)$.
    Following \cite{CY06}, 
    Let $\bar Z_\infty(B)$ denote the limit of $\bar Z_n(0)$ (which exists a.s.)
    and, for $z\in \R^d$,
    let $X_{n,z}=\bar Z_\infty(\theta_{n,z} B)/\E \bar Z_\infty$, where
    $\theta_{n,z}$ denote the temporal (by $n$) and spatial (by $z$) 
    shift of $B$.
    Set, for $x,z\in\R^d$,
    $$ e_{n,x,z}(B)= 
    E_x\left(
    \exp \left\{ \beta\int_{0}^{1}\int_{\R^d} 
    \phi(y-W_{s})\dot{B}(s+n-1, \d y) \d s-\frac{\beta^2 V(0)}{2}\right\}\big|
    W_{1}=z\right).$$
    We have that 
    $\E X_{n,z} =1$ and $X_{n,x}\geq E_x(e_{n+1,x,W_1}\cdot X_{n+1,W_1})$ 
    by Fatou.
Denote by 
$\mathcal{G}_t$ the natural filtration
induced by
$t\to B(t,\cdot)$.
    By construction, $X_{n,\cdot}$ is independent
    of $\mathcal{G}_n$, and $\E( X_{n,z}|\mathcal{G}_n)=\E X_{n,z}=1$ . Now,
    iterating, we get by the Markov property
    $$X_{0,0}\geq 
    E_0(e_{1,0,W_1}e_{2,W_1,W_2}\cdots e_{n,W_{n-1},W_n}X_{n,W_n})\,.$$
    Thus,
    $$\E(X_{0,0}|\mathcal{G}_n)\geq 
    E_0(e_{1,0,W_1}e_{2,W_1,W_2}\cdots e_{n,W_{n-1},W_n})=\bar Z_n\,.$$
    It follows that the sequence $\bar
    Z_n$ is uniformly integrable under $\P$.
 \end{proof}
 \begin{remark} 
     An alternative proof of Proposition \ref{prop-CY06} can be obtained
     by using \cite[Thm. 2]{KaCh} and an appropriate 0-1 law with respect to the
   Brownian path $W$.
 \end{remark}
The following proposition is the heart of the proof of Theorem \ref{theo-UT}.
\begin{prop}
  \label{prop-UT}
  There exists $\beta^*$ such that for $\beta>\beta^*$ and any $m>0$,
  $$ 
  \overline \Q_\eps(Z_\eps>m)\to_{\eps \to 0} 1. 
  $$ 
\end{prop}

We first complete the proof of Theorem \ref{theo-UT}, and then provide 
the proof of Proposition \ref{prop-UT}.

\noindent
{\bf{Proof of Theorem \ref{theo-UT}} (assuming Proposition \ref{prop-UT}):}
Assume that
$Z_\eps$ does not converge to $0$ almost surely. Then, 
  by Proposition \ref{prop-CY06},
it is uniformly integrable and, by Lemma 
\ref{lem-increase}, it is uniformly bounded in probability under $\overline \Q_\eps$. In particular, there exists $K>0$ such that
$\overline\Q_\eps(Z_\eps>K)<1/2$. This contradicts
Proposition \ref{prop-UT}.
\qed

Before providing the proof of Proposition \ref{prop-UT}, we need to introduce some notation and prove some preparatory lemmas.
Introduce 
the stopping times $\tau_\delta(W,W')=\inf\{t>0: |W_t-W'_t|\geq \delta\}$.
%
We need
an estimate on the tail of $\tau := \tau_\delta$ conditionally on $W$, presented in the
next lemma; in its statement and in 
its proof, $P_0^{\otimes 2}$ denotes
the measure $P_0\otimes P_0$ on $(W,W')$
\begin{lemma}\label{L4.4}
  \label{lem-tau}
  There exists a random variable $\chi=\chi(W)$ and a constant $\kappa>0$, such that for 
  $t$ large enough,
  $$
  P_0^{\otimes 2}
  \big(\tau\geq t|W\big) \geq \chi(W)e^{-\kappa t}\,.$$
\end{lemma}
\begin{proof}
Define
\begin{equation}\label{kappa1}
  \kappa_1 =\liminf_{t\to\infty} \frac1t\log  P_0^{\otimes 2}
  \big(\tau\geq t|W\big).
\end{equation}
Note that since $\kappa_1$ is measurable with respect to the tail $\sigma$-field
of $W^\prime$, it is deterministic, possibly equal to
$-\infty$. We will show that $\kappa_1>-\infty$. Taking then $\kappa= - 2\kappa_1$
then proves the lemma. 

With $|\cdot|$ denoting the Euclidean norm in $\R^d$, let
$$
W^{1,2}_t=\bigg\{\varphi: \varphi(0)=0, \int_0^t |\dot{\varphi}(s)|^2 \d s<\infty\bigg\},
$$
where
$\dot{\varphi}$ denotes the time-derivative of $\varphi$. 
We also use the notation $\|\varphi\|_{\infty,t}=\sup_{s\in [0,t]} |\varphi(s)|$.
Fix a (possibly random, but independent of $W'$) function $\varphi\in W^{1,2}_t$.
Then, by an application of the Cameron-Martin theorem in classical Wiener space,
\begin{eqnarray}
  \nonumber
  P_0(\|W^\prime-\varphi\|_{\infty,t}\leq \delta/2)&=&
  \int \e^{\int_0^t \dot{\varphi}(s)\d W^\prime(s)-\frac12 \int_0^t |\dot{\varphi}(s)|^2 \d s}  
  {\1}_{\{\|W^\prime\|_{\infty,t}\leq \delta/2\}} \d P_0(W^\prime)\\
  &=&
  \nonumber
 \e^{-\frac12 \int_0^t |\dot{\varphi}(s)|^2 \d s}  
  \int \e^{\int_0^t \dot{\varphi}(s) \d W^\prime(s) }  
  {\1}_{\{\|W^\prime\|_{\infty,t}\leq \delta/2\}} \d P_0(W^\prime)\\
  &=&
  \nonumber
  \e^{-\frac12 \int_0^t |\dot{\varphi}(s)|^2 \d s}  \,\,
  P_0\big(\|W'\|_{\infty,t}\leq \delta/2\big)
  \,\,  E_0\bigg[ \e^{\int_0^t \dot{\varphi}(s)\, \d W^\prime(s) } 
    \big| \big\{\|W^\prime\|_{\infty,t}\leq \delta/2\big\}\bigg]
  \\
  &\geq & 
  e^{-\frac12 \int_0^t |\dot{\varphi}(s)|^2 \d s}  
  P_0\big(\|W'\|_{\infty,t}\leq \delta/2\big),
  \label{eq-OM}
\end{eqnarray}
where the last inequality used Jensen's inequality and invariance of
the set $\|W'\|_{\infty,t}\leq \delta/2$ 
with respect to the map $W'\mapsto -W'$.

Introduce the random field
$$ 
Y_{s,t}(W)=
\inf\bigg\{\int_s^t |\dot{\varphi}(u)|^2 \,\, \d u: \varphi(s)=W_s,\varphi(t)=W_t, 
\sup_{u\in [s,t]} |W(u)-\varphi(u)|\leq \delta/2\bigg\}.
$$
Since $Y$ is subadditive in the sense that
$Y_{s,t}\leq Y_{s,u}+Y_{u,t}$
for $u\in (s,t)$, Kingman's  subadditive ergodic theorem 
implies that
\begin{equation}\label{kappa2}
  t^{-1} Y_{0,t} \to_{t\to \infty} \kappa_2,\quad a.s.
\end{equation}
for a deterministic $\kappa_2$. 
We claim that $\kappa_2$ is finite. This follows from
the fact that $\kappa_2$ is smaller
than $EY_{0,1}$; since $Y_{0,1}$ is finite almost surely and 
$X:=\sqrt{Y_{0,1}}$ is Lipshitz as a map on $\mathcal{E}$,
denoting by $\mbox{\rm med} (X)$ the median of $X$
we have by the Borell--Tsirelson-Ibragimov-Sudakov inequality 
\cite{AT07}
that $X-\mbox{\rm med}(X)$ possesses Gaussian tails, and
therefore $E X^2=EY_{0,1}<\infty$.

We can now conclude. Let $\varphi^{(t)}=\varphi^{(t)}(W)$ be such that
$\varphi^{(t)}(0)=0$, $\varphi^{(t)}(t)=W(t)$  and
$Y_{0,t}=\int_0^t |\dot{\varphi}(s)|^2 ds$. (Such $\varphi^{(t)}$ exists by
lower-semicontinuity of the $L^2$ norm, although this is not essential to our argument and we could just assume that the last integral is smaller than 
$2Y_{0,t}$.) We have, by \eqref{eq-OM},
$$
\begin{aligned}
  P^{\otimes 2}_0\big(\tau\geq t |W \big)=
  P^{\otimes 2}_0\big(\|W'-W\|_{\infty,t}\leq \delta |W\big)
  &\geq P^{\otimes 2}_0\big(\|W'-\varphi^{(t)}\|_{\infty,t}\leq \delta/2\big) \\
&\geq 
\e^{-\frac12 Y_{0,t} }
P_0(\|W^\prime\|_{\infty,t}\leq \delta/2\big) .
\end{aligned}
$$
Thus, by \eqref{kappa1} and \eqref{kappa2},
\begin{equation*}
  \kappa_1 =\liminf_{t\to\infty} \frac1t\log P^{\otimes 2}(\tau\geq t|W)
  \geq 
  -\frac {\kappa_2} 2
  +
  \lim_{t\to\infty} \frac1t\log 
  P_0(\|W'\|_{\infty,t}\leq \delta/2). 
\end{equation*}
The last probability on the right hand side is $P_0(\sigma >t)$,
where $\sigma$ denotes the first exit time of the standard Brownian motion $W^\prime$
from the ball of radius $\delta/2$ around the origin. It is well-known (for example, by the spectral theorem
for $-\frac 12 \Delta$) that $\lim_{t\to\infty}\frac 1t \log P_0\{\sigma>t\}=-\lambda_1$, where $\lambda_1>0$ is the
principal eigenvalue of $-\frac 12 \Delta$ with Dirichlet boundary conditions
on the same ball. It follows that $\kappa_1 >-\infty$ and
Lemma \ref{L4.4} is proved.
\end{proof}
Henceforth, we set  $t=\eps^{-2}$.
Next, on $\mathcal{E}\times \mathcal{E}$, introduce the kernels 
$$
K_\eps(W,W')=\int_0^{1/\eps^2}\int_{\R^d} \phi(x-W_s)\phi(x-W'_s) \d x \, \d s.
$$
Note 
that $K_\eps(W,W')\leq V(0) t$.
\begin{lemma}
  \label{lem-comp}
  There exists $\delta>0$ such that on the event
  $\{\tau_\delta(W,W')\geq t\}$, one has $K_\eps(W,W')\geq 2V(0)t/3$.
\end{lemma}
\begin{proof}
Note that $V(0)=\int_{\R^d} \phi^2(y) \d y$. On the other hand,
for $\theta$ small enough,
$$\inf_{f: \,\, \forall s, \,\,|f(s)|\leq \theta}\,\,  \int_0^t \int_{\R^d}
\phi(y) \phi\big (y+f(s)\big ) \d y \d s\geq t\big (V(0)-O(\theta)\big ).
$$
This completes the proof.
\end{proof}
Finally we turn to the proof of Proposition \ref{prop-UT}.

\noindent\textit{Proof of Proposition \ref{prop-UT}:}
Since we will use two independent copies $W,W'$ of Brownian motions,
we write throughout $\Lambda_\eps=\Lambda_\eps(W)$, $\Lambda_\eps(W')$ to 
emphasize which Brownian motion participates in
the definition of $\Lambda_\eps$.

The starting point of the proof
is the remark that by the Cameron-Martin change of measure 
\cite{Bogachev}, the law of $\dot B(x,s)$ under $\Q_\eps$
is the same as the law of $\dot B(x,s)+\beta \phi(x-W_s)$ under 
$\P\otimes P_0$.
In particular, for any measurable $A\subset \mathcal{E}$, the
law under $\overline\Q_\eps$ of $\int_A \Lambda_\eps(W^\prime)\d P_0(W^\prime)$
is the same as the law under $\P\otimes P_0$ of $\int_A \e^{\beta^2 K_\eps(W,W^\prime)}
  \Lambda_\eps(W') \d P_0(W^\prime)$.

  Let $f:\R_+\to \R_+$ be an increasing concave function. 
  Then, by the above remark,
  \begin{eqnarray}
    \label{array-1}
    \int f(Z_\eps) \d\overline\Q_\eps&= & \int f(Z_\eps) \d  \Q_\eps=
    \int f\bigg (\int \Lambda_\eps (W^\prime) \d P_0(W^\prime)\bigg)\; \, 
    \d \Q_\eps\nonumber\\
    &\geq &
    \int f\bigg(\int \Lambda_\eps (W^\prime)\,\, {\1}_{\{\tau(W,W')\geq t\}}\,\,
    \d P_0(W^\prime)\bigg) \d \Q_\eps\nonumber\\
    &= &
    \int f\bigg (\int \Lambda_\eps (W^\prime) \e^{\beta^2 K_\eps(W,W^\prime)}\,\, {\1}_{\{\tau(W,W')\geq t\}}
    \,\,\d P_0(W^\prime)\bigg ) \d (\P\otimes P_0) \nonumber \\
    &\geq &
    \int f\bigg(\int \Lambda_\eps (W^\prime) \e^{2\beta^2 V(0) t/3}\,\,
    {\1}_{\{\tau(W,W')\geq t\}} \,\,
   \d P_0(W^\prime)) \d (\P\otimes P_0) \nonumber \\
    &= &
    \int f\bigg(\e^{2\beta^2 V(0)t/3}\int \Lambda_\eps (W^\prime) \,\,
    {\1}_{\{\tau(W,W')\geq t\}} \,\,
    \d P_0(W')\bigg) \,\,\d (\P\otimes P_0),
  \end{eqnarray}
  where in the first inequality  we used that $f$ is increasing,
  and in the last inequality we used the same together with
  Lemma \ref{lem-comp} (recall $t=\eps^{-2}$). On the other hand, $f$ is concave and on the set $\{\tau \ge t\}$
  the covariance kernel $K_\eps$ is bounded from above by the 
  constant kernel $\widehat K_\eps(W, W^\prime) := V(0) t$.
  Using Kahane's comparison inequality with kernels $K_\eps$ 
  and $\widehat K_\eps$
  (see \cite{Kahane} -- it is stated there for \emph{convex} functions,
  with the opposite sign; see also \cite[Theorem 28]{shamov}), we get:
  \begin{equation}
    \label{eq-Kahane}
   \int f(Z_\eps) \d\overline\Q_\eps\geq 
    E_{G,W} \bigg[f\bigg(\e^{2\beta^2 V(0) t/3}\,\,\big(P_0\otimes P_0\big)
      \big(\tau(W,W')\geq t|W\big)\,\,\e^{\beta (V(0) t)^{1/2} G-
  \beta^2 V(0) t / 2}\bigg)\bigg],
\end{equation}
where $G$ is a standard centered Gaussian random variable which is independent of $W$,
and the expectation $E_{G,W}$ is taken over both $G$ and $W$. In particular,
  \begin{equation}
    \label{eq-Kahane-1}
    \begin{aligned}
   \int f(Z_\eps) \d\overline\Q_\eps &\geq 
   E_{G,W} \bigg[f\bigg(\e^{\beta^2 V(0)t/6}\,\,\big(P_0\otimes P_0\big)
     \big(\tau(W,W')>t|W\big)\e^{\beta (V(0)t)^{1/2} G}\bigg)\bigg] \\
   \\
   &\geq 
   E_{G,W}\bigg[f\bigg(\chi(W) e^{-\kappa t}
   \e^{\beta^2 V(0)t/6}e^{\beta \sqrt{V(0)t} G}\bigg)\bigg].
   \end{aligned}
 \end{equation}
 Note that the argument of $f$ goes to infinity  as $t\to\infty$ for
 almost every $(G,W)$, if $\beta>\sqrt{6\kappa}$. Using 
 $$
f(x)=f_\alpha(x)= \left\{\begin{array}{ll}
  \alpha^{-1} x, & x\leq \alpha\\
  1,& x\geq \alpha,
\end{array}\right.$$
we conclude that
$$ \lim_{\alpha\to\infty} \liminf_{\eps\to 0}
\int f_\alpha(Z_\eps) \d\overline\Q_\eps =1.$$
This completes the proof.\qed

\noindent\textit{Proof of Corollary \ref{cor-beta}.}
Recall the random variable 
$$Z_\eps=Z_{\eps,\beta}(B)= E_0\bigg[
\exp\bigg\{\beta \int_0^{\eps^{-2}} \int_{\R^d} \phi\big({y-W_s}\big)\dot B(s, \d y)
\d s - \frac {\beta^2}{2\eps^2} V(0)\bigg\}\bigg].
$$
Let 
$$
\overline \beta=\sup\bigg\{\beta>0: \big\{Z_{\eps,\beta}\big\}_{\eps>0}
\;	\mbox{\rm is uniformly integrable}\bigg\}.
$$
In view of Theorem \ref{thm1} and Theorem \ref{theo-UT}, we have
$\overline\beta\in (0,\infty)$.
Thus, the corollary will follow from the following fact.
\begin{align}
	\label{eq-cor1}
	&\mbox{\rm	If $Z_{\eps,\beta}$ is uniformly 
	integrable for some $\beta>0$,
	then
	so is $Z_{\eps,\beta'}$ for $\beta'<\beta$.}
\end{align}

To see \eqref{eq-cor1}, let $B,B'$ be independent copies of $B$ and
let $\beta'=\rho \beta$ with $\rho<1$. To emphasize the dependence of
$Z_{\eps,\beta}$ on $B$, 
we write $Z_{\eps,\beta}=Z_{\eps,\beta}(B)$.
Note  that
$$
Z_{\eps,\beta^\prime}(B)=Z_{\eps,\rho \beta}(B)=
\E\left[Z_{\eps,\beta}(\rho B+\sqrt{1-\rho^2} B')\left|B\right.\right]
$$
Since $\big\{Z_{\eps,\beta}(B)\}_{\eps>0}$ is uniformly integrable,
there exists a positive increasing convex function $f$ with 
$f(x)/x\to_{x\to\infty} \infty$ so that $\sup_\eps
\E f(Z_{\eps,\beta}(B))<\infty$. 
However, by Jensen's inequality and the last display,
$$
\begin{aligned}
\E [f(Z_{\eps,\beta'}(B))]&=
\E \bigg[f\left(\E\left(Z_{\eps,\beta}(\rho B+
\sqrt{1-\rho^2} B')\left|\right.B\right)
\right)\bigg]\\
&\leq  \E \big[f(Z_{\eps,\beta}(\rho B+\sqrt{1-\rho^2} B'))\big]
=
\E [f(Z_{\eps,\beta}(B))]\,.
\end{aligned}
$$
It follows that $\sup_{\eps>0} \E [f(Z_{\eps,\beta'}(B))]<\infty$, which
in turn implies the uniform integrability of $\big\{Z_{\eps,\beta'}\big\}_{\eps>0}$. This completes
the proof.
\qed

{\bf{Acknowledgments.}} The first author would like to thank Herbert Spohn (Munich) for suggesting a problem that led to this study, 
and for inspiring discussions.

\end{document}